\documentclass[a4, 12pt]{amsart}
\usepackage[mathscr]{eucal}
\usepackage{amssymb}
\usepackage{latexsym}
\usepackage{amsthm}
\theoremstyle{plain}
\newtheorem{theorem}{Theorem}[section]
\newtheorem{corollary}{Corollary}[section]
\newtheorem{remark}{Remark}[section]
\newtheorem{lemma}{Lemma}[section]

\makeatletter
\@addtoreset{equation}{section}

\title[Upper and lower  bounds for eigenvalues ]{Upper and lower bounds for eigenvalues \\
 of the clamped plate problem*}
\author{Qing-Ming Cheng and  Guoxin Wei}
\address{Qing-Ming Cheng \\  Department of Mathematics, Graduate School  of Science and Engineering,
Saga University, Saga 840-8502,  Japan, cheng@ms.saga-u.ac.jp}
\address{Guoxin Wei \\  School of Mathematical Sciences, South China Normal University,
510631, Guangzhou,  China, weiguoxin@tsinghua.org.cn}

\begin{document}
\maketitle

\begin{abstract}
\noindent In this paper,  we study  estimates for eigenvalues of the clamped plate problem.
A sharp upper bound for eigenvalues  is given and
the lower bound for eigenvalues in  \cite{[CW]} is  improved.
\end{abstract}

\footnotetext{ 2001 \textit{ Mathematics Subject Classification}: 35P15.}

\footnotetext{{\it Key words and phrases}: The clamped plate problem, upper bounds for eigenvalues, lower bounds for eigenvalues.}

\footnotetext{* The first author was partly supported by a Grant-in-Aid for
Scientific Research from JSPS. The second author was partly supported by grant No. 11001087 of NSFC.}

\section {Introduction}
\noindent
A  membrane has its transverse vibration  governed by  equation
\begin{equation*}
      \Delta u =-\lambda u, \  \ {\rm in} \ \ \Omega
\end{equation*}
with the boundary condition
$$
u=0 , \ \ {\rm on}  \ \ \partial \Omega,
$$
where $\Delta$ is the Laplacian in  $\mathbf{R}^n$ and $\Omega$ is a bounded domain in
$\mathbf{R}^n$. It is classical that there is a countable sequence of eigenvalues
$$
0<\lambda_1<\lambda_2\leq\lambda_3\leq\cdots\rightarrow\infty,
$$
and a sequence of corresponding eigenfunctions $u_1, u_2, \cdots, u_k, \cdots$
such that
\begin{equation*}
      \Delta u_k =-\lambda_k u_k, \  \ {\rm in} \ \ \Omega.
\end{equation*}
The eigenfunctions form an orthonormal basis of $L^2(\Omega)$.

On the other hand,  the vibration of a stiff plate differs from that of a membrane
not only in the equation which governs its motion
but also in the way the plate is fastened to its boundary.
A plate spanning a domain $\Omega$ in  $\mathbf{R}^n$ has its transverse vibrations governed
by
\begin{equation}\label{eq:1.1}
  {\begin{cases}
     \Delta^2 u = \Gamma  u,& \ \ {\rm in} \ \ \Omega ,\\
     u=\dfrac{\partial u}{\partial \nu}=0 , & \ \ {\rm on}  \ \ \partial \Omega,
     \end{cases}}
  \end{equation}
where $\nu$ denotes the outward unit normal to the boundary $\partial \Omega$.
Namely, not only  is the rim of the plate firmly fastened to the boundary, but the plate
is clamped so that lateral motion can occur at the edge. One calls it  {\it a clamped plate problem}.
It is known that this problem has a real and discrete spectrum
$$
0<\Gamma_1\leq\Gamma_2\leq\cdots\leq\Gamma_k\leq\cdots\to
+\infty,$$ where each $\Gamma_i$ has finite multiplicity which is
repeated according to its multiplicity.

\noindent
For the eigenvalues of the clamped plate problem \eqref{eq:1.1},  Agmon  \cite{[Ag]} and Pleijel \cite{[P1]}
gave the following  asymptotic  formula,
$$
\Gamma_k\sim
\dfrac{16\pi^4}{\big(B_n\text{vol}(\Omega)\big)^{\frac{4}{n}}}k^{\frac{4}{n}},\
\ \  k\rightarrow\infty.
  $$
This implies that
\begin{equation}\label{eq:1.2}
\frac{1}{k}\sum_{j=1}^k\Gamma_j
\sim\frac{n}{n+4}\dfrac{16\pi^4}{\big(B_n\text{vol}(\Omega)\big)^{\frac{4}{n}}}k^{\frac{4}{n}},
\ \ k\rightarrow\infty,
\end{equation}
where $B_n$ denotes the volume of the unit ball in $\mathbf R^n$.
Furthermore, Levine and Protter \cite{[LP]} proved that the eigenvalues of
the clamped plate problem \eqref{eq:1.1} satisfy
$$\dfrac{1}{k}\sum_{j=1}^k\Gamma_j
\geq\frac{n}{n+4}\dfrac{16\pi^4}{\big(B_n\text{vol}(\Omega)\big)^{\frac{4}{n}}}k^{\frac{4}{n}}.
$$
The formula \eqref{eq:1.2} shows that the coefficient of $k^{\frac{4}{n}}$ is the best possible constant.
Thus, it will be interesting and very important to find  the second term on $k$ of the asymptotic expansion formula
of $\Gamma_k$.
The authors \cite{[CW]}  have made effort for this problem. We have improved  the result due to Levine and Protter \cite{[LP]}  by adding to its right hand side two terms of lower order in $k$:
\begin{equation}\label{eq:1.3}
\begin{aligned}
& \frac{1}{k}\sum_{j=1}^k\Gamma_j\geq
\frac{n}{n+4}\dfrac{16\pi^4}{\big(B\text{vol}(\Omega)\big)^{\frac{4}{n}}}k^{\frac{4}{n}}\\
 &+\left(\frac{n+2}{12n(n+4)}-\frac{1}{1152n^2(n+4)}
    \right)\frac{\text{vol}(\Omega)}{I(\Omega)}\frac{n}{n+2}
 \dfrac{4\pi^2}{\big(B\text{vol}(\Omega)\big)^{\frac{2}{n}}}k^{\frac{2}{n}}\\
 &+\left(\frac{1}{576n(n+4)}-\frac{1}{27648n^2(n+2)(n+4)}\right)
 \left(\frac{\text{vol}(\Omega)}{I(\Omega)}\right)^2,
\end{aligned}
\end{equation}
where
$$
I(\Omega)=\min\limits_{a\in \mathbf{R}^n}\int_\Omega|x-a|^2dx
$$
 is called {\it  the moment of inertia} of $\Omega$.
On the other hand, if one
can obtain an upper bound with optimal order of $k$ for eigenvalue $\Gamma_k$, then one can know the
exact  second term on $k$. From our knowledge, there is no any result on upper bounds for eigenvalue
$\Gamma_k$ with optimal order of $k$.  In \cite{CY2}, Cheng and Yang have established a recursion
formula in order to obtain  upper bounds for eigenvalues of the Dirichlet eigenvalue problem of the
Laplacian. Hence, if one can get a sharper universal inequality for eigenvalues of the clamped plate
problem, we can also derive an upper bound for eigenvalue $\Gamma_k$ by making use of the recursion
formula due to Cheng and Yang \cite{CY2}. On the investigation of universal inequalities for eigenvalues of the clamped plate problem,
Payne, P\'olya and Weinberger \cite{[PPW]} proved
\begin{equation*}
\Gamma_{k+1} - \Gamma_{k} \leq \frac{8(n+2)}{n^{2} k} \sum_{i=1}^{k} \Gamma_{i}.
\end{equation*}
 Chen and Qian \cite{CQ} and Hook \cite{H}, independently, extended
 the above inequality to
\begin{equation*}
\frac{n^{2} k^{2}}{8(n+2)} \leq
 \sum_{i=1}^{k} \frac{\Gamma_{i}^{\frac{1}{2}}}{\Gamma_{k+1}
 - \Gamma_{i}} \sum_{i=1}^{k} \Gamma_{i}^{\frac{1}{2}}.
\end{equation*}
Recently, answering a question of Ashbaugh \cite{Ash},  Cheng and Yang \cite{CY2} have proved the
following remarkable estimate:
\begin{equation*}
\sum_{i=1}^{k} (\Gamma_{k+1} - \Gamma_{i} ) \leq (\frac{8(n+2)}{n^{2}})^{\frac{1}{2}} \sum_{i=1}^{k} (\Gamma_{i} (\Gamma_{k+1} - \Gamma_{i} ))^{\frac{1}{2}}.
\end{equation*} \\
Furthermore, Wang and Xia \cite{WX} (cf. Cheng, Ichikawa and Mametsuka \cite{CIM1}, \cite{CIM2}) have proved
\begin{equation*}
\sum_{i=1}^k (\Gamma_{k+1}-\Gamma_{i})^2 \leq \displaystyle{\frac{8(n+2)}{n^2}}\sum_{i=1}^k(\Gamma_{k+1}-\Gamma_{i})\Gamma_{i}.
\end{equation*}
The first author has conjectured the following:

\vskip 2mm
\noindent
{\bf Conjecture.} Eigenvalue $\Gamma_j$'s  of  the clamped plate problem (1.1) satisfy
\begin{equation}\label{eq:1.4}
\sum_{j=1}^k (\Gamma_{k+1}-\Gamma_{j})^2 \leq \displaystyle{\frac{8}{n}}\sum_{j=1}^k(\Gamma_{k+1}-\Gamma_{j})\Gamma_{j}.
\end{equation}
If one can solve the above conjecture, then from the recursion formula of Cheng and Yang \cite{CY2},
we can derive an upper bound for the eigenvalue $\Gamma_k$ with the optimal order of $k$. But it
seems to be hard to solve this conjecture.

In this paper, we will try to use a fact that  eigenfunctions of the clamped plate problem \eqref{eq:1.1}
form an orthonormal basis of  the Sobolev Space $W_0^{2,2}(\Omega)$
to get an upper bound for eigenvalues of the clamped plate problem \eqref{eq:1.1}.
A similar fact for the Dirichlet eigenvalue problem of the Laplacian  is also
used by Li and Yau \cite{LY} and Kr\"oger \cite{K}.
Furthermore, we will give an improvement of  the inequality \eqref{eq:1.3}.

Let $\Omega$ be a bounded domain with a smooth boundary $\partial\Omega$ in
the  $n$-dimensional Euclidean space $\mathbf{R}^n$. Let $d(x)=\text{dist}(x, \partial\Omega)$ denote
the distance function from the point $x$ to the boundary $\partial \Omega$ of $\Omega$.
We define
$$
\Omega_r=\biggl\{x\in\Omega \  | \ \ d(x)<\dfrac1r\biggl\}.
$$

\begin{theorem}
Let $\Omega$ be a bounded domain with a smooth boundary $\partial \Omega$ in $\mathbf {R}^n$.
Then there exists a constant $r_0>0$ such that eigenvalues of the clamped plate problem {\rm \eqref{eq:1.1}} satisfy
\begin{equation}\label{eq:1.5}
\begin{aligned}
&\dfrac{1}{k}\sum_{j=1}^{k}\Gamma_j
\leq\dfrac{1+\dfrac{4(n+4)(n^2+2n+6)}{n+2}\dfrac{\text{vol}(\Omega_{r_0}) }{\text{vol}(\Omega)}}{\bigl(1-\dfrac{\text{vol}(\Omega_{r_0}) }{\text{vol}(\Omega)}\bigl)^{\frac{n+4}n}}\dfrac n{n+4}\dfrac{16\pi^4}{\bigl(B_n\text{vol}(\Omega)\bigl)^{\frac4n}}k^{\frac4n},\\
\end{aligned}
\end{equation}
for $k\geq \text{vol}(\Omega)r_0^n$.
\end{theorem}
\begin{remark}
Since  $\text{vol}(\Omega_{r_0})\to 0$ when $r_0\to \infty$, we know that  the upper bound in the theorem 1.1
is sharp in the sense of the  asymptotic formula due to Agmon and Pleijel.
\end{remark}

\begin{corollary}
Let $\Omega$ be a bounded domain with a smooth boundary $\partial \Omega$ in $\mathbf {R}^n$.
If there exists a constant $c_0$ such that
$$
\text{vol}(\Omega_r)\leq c_0\text{vol}(\Omega)^{\frac{n-1}n}\dfrac1r
$$
for $r>\text{vol}(\Omega)^{\frac{-1}n}$,
then there exists a constant $r_0$ such that eigenvalues of the clamped plate problem {\rm \eqref{eq:1.1}} satisfy
\begin{equation}\label{eq:1.6}
\begin{aligned}
&\dfrac{1}{k}\sum_{j=1}^{k}\Gamma_j
\leq \dfrac n{n+4}\dfrac{16\pi^4}{\bigl(B_n\text{vol}(\Omega)\bigl)^{\frac4n}}\biggl( k^{\frac4n}
+c_0c(n) k^{\frac3n}\biggl),
\end{aligned}
\end{equation}
for $k=\text{vol}(\Omega)r_0^n>c_0^n$, where $c(n)$ is  a constant depended only on $n$.
\end{corollary}

\begin{theorem}  Let $\Omega$ be a bounded domain
with a piecewise smooth boundary $\partial \Omega$ in $\mathbf {R}^n$.
Eigenvalue $\Gamma_j$'s  of the clamped plate problem {\rm \eqref{eq:1.1}} satisfy
\begin{equation}\label{eq:1.6}
\begin{aligned}
 \frac{1}{k}\sum_{j=1}^k\Gamma_j\geq&
\frac{n}{n+4}\dfrac{16\pi^4}{\big(B_n\text{vol}(\Omega)\big)^{\frac{4}{n}}}k^{\frac{4}{n}}\\
 &+\frac{n+2}{12n(n+4)}\frac{\text{vol}(\Omega)}{I(\Omega)}\frac{n}{n+2}
 \dfrac{4\pi^2}{\big(B_n\text{vol}(\Omega)\big)^{\frac{2}{n}}}k^{\frac{2}{n}}\\
 &+\frac{(n+2)^2}{1152n(n+4)^2}
 \left(\frac{\text{vol}(\Omega)}{I(\Omega)}\right)^2,
\end{aligned}
\end{equation}
where $I(\Omega)$ is the moment of inertia  of $\Omega$.
\end{theorem}

\vskip 5mm

\section{Upper bounds for eigenvalues}
\noindent
In this section, we will study the upper bounds for eigenvalues of
the clamped plate problem \eqref{eq:1.1}.

\vskip 2mm
\noindent
{\it Proof of  Theorem 1.1}.
Since $d(x)$ is the distance function from the point $x$ to the boundary $\partial \Omega$
of $\Omega$, we define a function $f_r$ for any fixed $r$ by
\begin{equation}\label{eq:2.1}
 f_r(x)=\begin{cases}
      1, & \ x\in \Omega, d(x)\geq \frac1r, \\
     r^2d^2(x), & \   x\in \Omega, d(x)<\frac1r,  \\
    0,   & \ \ \text{the other}.
     \end{cases}
\end{equation}
Let $u_j$ be an orthonormal eigenfunction
corresponding to the eigenvalue $\Gamma_j$, that is,  $u_j$ satisfies
\begin{equation*}
  {\begin{cases}
     \Delta^2 u_j = \Gamma_j  u_j,& \ \ {\rm in} \ \ \Omega ,\\
     u_j=\dfrac{\partial u_j}{\partial \nu}= 0 , & \ \ {\rm on}  \ \ \partial \Omega, \\
     \int_\Omega u_i(x)u_j(x)dx=\delta_{ij}, & \ \ \text{for any $i$, $j$}.
     \end{cases}}
\end{equation*}
Thus, $\{u_j\}$ forms an orthonormal basis of the Sobolev Space $W_0^{2,2}(\Omega)$.
For an arbitrary fixed point  $z\in\mathbf{R}^n$ and $r>0$,
a function
\begin{equation}\label{eq:2.2}
g_{r,z}(x)=e^{i\langle z,x\rangle}f_r(x),
\end{equation}
with $i=\sqrt{-1}$, belongs  to the Sobolev Space $W_0^{2,2}(\Omega)$.
Hence, we have
\begin{equation}\label{eq:2.3}
g_{r,z}(x)=\sum_{j=1}^{\infty}a_{r, j}(z)u_j(x),
\end{equation}
where
\begin{equation}\label{eq:2.4}
a_{r,j}(z)=\int_{\Omega}g_{r,z}(x)u_j(x)dx.
\end{equation}
Defining  a function
\begin{equation}\label{eq:2.5}
\varphi_k(x)=g_{r,z}(x)-\sum_{j=1}^{k}a_{r, j}(z)u_j(x),
\end{equation}
we have $\varphi_k=\dfrac{\partial \varphi_k}{\partial \nu}=0$ on $\partial\Omega$ and
$$
\int_{\Omega}\varphi_k(x)u_j(x)dx=0, \ \ \text{for} \  j=1, 2, \cdots, k.
$$
Therefore, $\varphi_k$ is a trial function. From Rayleigh-Ritz formula, we have
\begin{equation}\label{eq:2.6}
\Gamma_{k+1}\int_\Omega|\varphi_k(x)|^2dx\leq\int_\Omega |\Delta \varphi_k(x)|^2dx.
\end{equation}
From the definition of $\varphi_k$ and \eqref{eq:2.1}, we have
\begin{equation}\label{eq:2.7}
\begin{aligned}
\int_\Omega|\varphi_k(x)|^2dx&=\int_{\Omega}|g_{r,z}(x)-\sum_{j=1}^{k}a_{r, j}(z)u_j(x)|^2\\
&=\int_{\Omega}|f_{r}(x)|^2dx-\sum_{j=1}^{k}|a_{r, j}(z)|^2\\
&\geq \text{vol}(\Omega)-\text{vol}(\Omega_r)-\sum_{j=1}^{k}|a_{r, j}(z)|^2.
\end{aligned}
\end{equation}
From \eqref{eq:2.5} and Stokes' formula, we infer
\begin{equation}\label{eq:2.8}
\begin{aligned}
&\int_\Omega|\Delta \varphi_k(x)|^2dx=\int_\Omega|\Delta g_{r,z}(x)-\sum_{j=1}^{k}a_{r, j}(z)\Delta u_j(x)|^2dx\\
&=\int_\Omega\biggl(|\Delta g_{r,z}(x)|^2+|\sum_{j=1}^{k}a_{r, j}(z)\Delta u_j(x)|^2\biggl)dx\\
&\ \ \ -\int_\Omega\biggl(\Delta g_{r,z}(x)\sum_{j=1}^{k}\overline{a_{r, j}(z)}\Delta u_j(x)
+\overline{\Delta g_{r,z}(x)}\sum_{j=1}^{k} a_{r, j}(z)\Delta u_j(x)\biggl)dx\\
&=\int_\Omega|\Delta g_{r,z}(x)|^2dx-\sum_{j=1}^{k}\Gamma_j|a_{r, j}(z)|^2\\
&=\int_\Omega\biggl|-|z|^2f_r(x)+2i\langle z,\nabla f_r(x)\rangle +\Delta f_{r}(x)\biggl|^2dx
-\sum_{j=1}^{k}\Gamma_j|a_{r, j}(z)|^2\\
&=\int_\Omega\biggl\{\biggl(-|z|^2f_r(x) +\Delta f_{r}(x)\biggl)^2+4\langle z,\nabla f_r(x)\rangle^2\biggl\}dx
-\sum_{j=1}^{k}\Gamma_j|a_{r, j}(z)|^2\\
\end{aligned}
\end{equation}
since
$$
\Delta g_{r,z}(x)=e^{i\langle z,x\rangle}\biggl(-|z|^2f_r(x)+2i\langle z,\nabla f_r(x)\rangle +\Delta f_{r}(x)\biggl).
$$
According to the definition of the function $f_r$, we have
\begin{equation*}
 \Delta f_r(x)=\begin{cases}
      0, & \ x\in \Omega, d(x)\geq \frac1r, \\
     r^2\Delta d^2(x), & \   x\in \Omega, d(x)<\frac1r,  \\
    0,   & \ \ \text{the other}.
     \end{cases}
\end{equation*}
Hence, we obtain, from the Schwarz inequality and $|\nabla d(x)|^2=1$,
\begin{equation}\label{eq:2.9}
\begin{aligned}
&\int_\Omega\biggl\{\biggl(-|z|^2f_r(x) +\Delta f_{r}(x)\biggl)^2+4\langle z,\nabla f_r(x)\rangle^2\biggl\}dx\\
&\leq |z|^4\text{vol}(\Omega)+24r^2|z|^2\text{vol}(\Omega_r)+\int_{\Omega_r}\biggl(\Delta f_{r}(x)\biggl)^2dx.
\end{aligned}
\end{equation}
 For a point $x\in \Omega$, there is a point $y=y(x)\in \partial \Omega$
such that
$
d(x)={\text {dist}}(x,y),
$
then we know that \begin{equation}\label{eq:2.16}
\Delta d^2(x)=2n-\sum_{j=1}^{n-1}\dfrac{2}{1-\kappa_jd(x)},
\end{equation}
where $\kappa_1, \kappa_2, \cdots, \kappa_{n-1}$ are the principal curvatures of $\partial\Omega$
at the point $y$. Since the boundary $\partial \Omega$ of the domain $\Omega$ is smooth and a compact hypersurface, one has that all of $\kappa_j$ are bounded. Without loss of generality, we can assume that $|\kappa_j(y)|\leq \kappa$ for any $y\in \partial\Omega$, $1\leq j\leq n-1$, then it follows that if $r\geq r_0>n\kappa$, then we see from \eqref{eq:2.16}
$$
0<\Delta d^2(x)<2n,\ \ x\in \Omega_r
$$
and
$$
\int_{\Omega_r}\biggl(\Delta f_{r}(x)\biggl)^2dx\leq 4n^2r^4\text{vol}(\Omega_r).
$$
Hence, if $r>r_0$, then we obtain
\begin{equation}\label{eq:2.10}
\begin{aligned}
&\int_\Omega|\Delta \varphi_k(x)|^2dx\\
&\leq |z|^4\text{vol}(\Omega)+24r^2|z|^2\text{vol}(\Omega_r)
+4n^2r^4\text{vol}(\Omega_r)-\sum_{j=1}^{k}\Gamma_j|a_{r, j}(z)|^2.
\end{aligned}
\end{equation}
From \eqref{eq:2.6}, \eqref{eq:2.7} and \eqref{eq:2.10}, we have
\begin{equation}\label{eq:2.11}
\begin{aligned}
&\Gamma_{k+1}\bigl(\text{vol}(\Omega)-\text{vol}(\Omega_r)\bigl)\\
&\leq |z|^4\text{vol}(\Omega)+24r^2|z|^2\text{vol}(\Omega_r)
+4n^2r^4\text{vol}(\Omega_r)+\sum_{j=1}^{k}(\Gamma_{k+1}-\Gamma_j)|a_{r, j}(z)|^2,
\end{aligned}
\end{equation}
here $r>r_0$.

\noindent Let  $B_n(r)$ denote the ball with a radius $r$ and the origin o in $\mathbf{R}^n$.
By integrating the above inequality on the  variable $z$ on the ball $B_n(r)$,  we derive
\begin{equation}\label{eq:2.12}
\begin{aligned}
&r^nB_n\bigl(\text{vol}(\Omega)-\text{vol}(\Omega_r)\bigl)\Gamma_{k+1}\\
&\leq r^{n+4}B_n\biggl(\dfrac n{n+4}\text{vol}(\Omega)+24\dfrac{n}{n+2}\text{vol}(\Omega_r)
+4n^2\text{vol}(\Omega_r)\biggl)\\
&\ \ \ +\sum_{j=1}^{k}(\Gamma_{k+1}-\Gamma_j)\int_{B_n(r)}|a_{r, j}(z)|^2dz,\ \ \ \ r>r_0.
\end{aligned}
\end{equation}
From Parseval's identity  for Fourier transform,  we have
\begin{equation}\label{eq:2.13}
\begin{aligned}
&\int_{B_n(r)}|a_{r, j}(z)|^2dz\leq \int_{\mathbf{R}^n}|a_{r, j}(z)|^2dz\\
&=\int_{\mathbf{R}^n}\bigl|\int_{\mathbf{R}^n}e^{i\langle z,x\rangle}f_r(x)u_j(x)dx\bigl|^2dz\\
&=(2\pi)^n\int_{\mathbf{R}^n}\bigl|{\widehat{f_ru_j}}(z)\bigl|^2dz
=(2\pi)^n\int_{\mathbf{R}^n}\bigl|f_r(x)u_j(x)\bigl|^2dx\\
&\leq (2\pi)^n.
\end{aligned}
\end{equation}
We obtain
\begin{equation}\label{eq:2.14}
\begin{aligned}
&r^nB_n\bigl(\text{vol}(\Omega)-\text{vol}(\Omega_r)\bigl)\Gamma_{k+1}\\
&\leq r^{n+4}B_n\biggl(\dfrac n{n+4}\text{vol}(\Omega)+24\dfrac{n}{n+2}\text{vol}(\Omega_r)
+4n^2\text{vol}(\Omega_r)\biggl)\\
&\ \ \ +(2\pi)^n\sum_{j=1}^{k}(\Gamma_{k+1}-\Gamma_j),\ \ \ \ r>r_0.
\end{aligned}
\end{equation}
Taking $r=2\pi\biggl(\dfrac{1+k}{B_n\bigl(\text{vol}(\Omega)-\text{vol}(\Omega_{r_0})\bigl)}\biggl)^{\frac1n}$,  noting $k\geq \text{vol}(\Omega)r_0^n$ and $\frac{2\pi}{(B_n)^{\frac{1}{n}}}>1$, then we can obtain $r>r_0$ and
\begin{equation}\label{eq:2.15}
\begin{aligned}
&\dfrac{1}{1+k}\sum_{j=1}^{k+1}\Gamma_j\\
&\leq16\pi^4\dfrac{\dfrac n{n+4}\text{vol}(\Omega)+\dfrac{24n}{n+2}\text{vol}(\Omega_r)
+4n^2\text{vol}(\Omega_r)}{(\text{vol}(\Omega)-\text{vol}(\Omega_{r_0})\bigl)^{\frac{n+4}n}}\dfrac{1}{B_n^{\frac4n}}(1+k)^{\frac4n}\\
&\leq16\pi^4\dfrac{\dfrac n{n+4}+(24\dfrac{n}{n+2}+4n^2)\dfrac{\text{vol}(\Omega_{r_0}) }{\text{vol}(\Omega)}}{\bigl(1-\dfrac{\text{vol}(\Omega_{r_0}) }{\text{vol}(\Omega)}\bigl)^{\frac{n+4}n}}\dfrac{1}{\bigl(B_n\text{vol}(\Omega)\bigl)^{\frac4n}}(1+k)^{\frac4n}
\end{aligned}
\end{equation}
This completes the proof of Theorem 1.1.
$$\eqno{\Box}$$

\noindent
{\it Proof of the corollary 1.1}.
From \eqref{eq:2.15} we have
\begin{equation}
\begin{aligned} \dfrac{1}{1+k}\sum_{j=1}^{k+1}\Gamma_j
\leq\dfrac{1+4(\frac{6}{n+2}+n)(n+4)\dfrac{\text{vol}(\Omega_{r}) }{\text{vol}(\Omega)}}{\bigl(1-\dfrac{\text{vol}(\Omega_{r_0}) }{\text{vol}(\Omega)}\bigl)^{\frac{n+4}n}}\dfrac n{n+4}\dfrac{16\pi^4}{\bigl(B_n\text{vol}(\Omega)\bigl)^{\frac4n}}(1+k)^{\frac4n}.\\
\end{aligned}
\end{equation}
Since $r=2\pi\biggl(\dfrac{1+k}{B_n\bigl(\text{vol}(\Omega)-\text{vol}(\Omega_{r_0})\bigl)}\biggl)^{\frac1n}$, we have
$$
\dfrac{\text{vol}(\Omega_{r})}{\text{vol}(\Omega)}\leq c_0\dfrac{B_n^{\frac1n}}{2\pi}\bigl(1-\dfrac{\text{vol}(\Omega_{r_0}) }{\text{vol}(\Omega)}\bigl)^{\frac{1}n} (1+k)^{-\frac1n}.
$$
Taking $c_1=4(\frac{6}{n+2}+n)(n+4)\dfrac{B_n^{\frac1n}}{2\pi}c_0$, we have
\begin{equation}
\begin{aligned}
&\dfrac{1}{1+k}\sum_{j=1}^{k+1}\Gamma_j\\
&\leq\dfrac{1+c_1\bigl(1-\dfrac{\text{vol}(\Omega_{r_0}) }{\text{vol}(\Omega)}\bigl)^{\frac{1}n}(1+k)^{-\frac1n}}{\bigl(1-\dfrac{\text{vol}(\Omega_{r_0}) }{\text{vol}(\Omega)}\bigl)^{\frac{n+4}n}}\dfrac n{n+4}\dfrac{16\pi^4}{\bigl(B_n\text{vol}(\Omega)\bigl)^{\frac4n}}(1+k)^{\frac4n}.\\
\end{aligned}
\end{equation}
Since there exists a constant $\alpha$ such that
$$
0<v=\dfrac{\text{vol}(\Omega_{r_0}) }{\text{vol}(\Omega)}\leq \dfrac{c_0}{(1+k)^{\frac1n}}\leq \alpha <1
$$
with $r_0=\biggl(\dfrac{1+k}{\text{vol}(\Omega)}\biggl)^{\frac1n}$, we define  a function
$$
G(v)=\dfrac{1+c_1(1-v)^{\frac1n}(1+k)^{-\frac1n}}{(1-v)^{\frac{n+4}n}}
$$
with $G(0)=1+c_1(1+k)^{-\frac1n}$. Since
$$
G^{\prime}(v)=\dfrac{1+\dfrac4n+c_1(1+\dfrac3n)(1-v)^{\frac1n}(1+k)^{-\frac1n}}{(1-v)^{\frac{2n+4}n}},
$$
by Lagrange mean value theorem, there exists
 $0<\theta<1$ such that
$$G(v)=G(0)+G^{\prime}(\theta v)v.$$
Hence, there exists a constant $c(n)$ only depended on $n$ such that
\begin{equation*}
\begin{aligned}
G(v)=&G(0)+G^{\prime}(\theta v)v\\
    =&1+c_1(1+k)^{-\frac1n}+\dfrac{1+\dfrac4n+c_1(1+\dfrac3n)(1-\theta v)^{\frac1n}(1+k)^{-\frac1n}}{(1-\theta v)^{\frac{2n+4}n}} v\\
    \leq &1+c_1(1+k)^{-\frac1n}+\dfrac{1+\dfrac4n+c_1(1+\dfrac3n)}{(1-\theta \alpha)^{\frac{2n+4}n}} c_0(1+k)^{-\frac1n}\\
    \leq& 1+c_0c(n)(1+k)^{-\frac{1}{n}},
\end{aligned}
\end{equation*}
that is,
 $$
 \dfrac{1+c_1\bigl(1-\dfrac{\text{vol}(\Omega_{r_0}) }{\text{vol}(\Omega)}\bigl)^{\frac{1}n}(1+k)^{-\frac1n}}{\bigl(1-\dfrac{\text{vol}(\Omega_{r_0}) }{\text{vol}(\Omega)}\bigl)^{\frac{n+4}n}}
 \leq 1+c_0c(n)(1+k)^{-\frac1n}.
 $$
Therefore, we obtain
\begin{equation*}
\begin{aligned} \dfrac{1}{1+k}\sum_{j=1}^{k+1}\Gamma_j
\leq \dfrac n{n+4}\dfrac{16\pi^4}{\bigl(B_n\text{vol}(\Omega)\bigl)^{\frac4n}}
\biggl((1+k)^{\frac4n}+c_0c(n)(1+k)^{\frac3n}\biggl).\\
\end{aligned}
\end{equation*}
This finishes the proof of the corollary 1.1.
$$\eqno{\Box}$$

\vskip 5mm
\section{Lower bounds for eigenvalues}

\noindent
In this section, we will give a proof of the theorem 1.2.
The following  lemma 3.1 will play an important role in the proof of theorem 1.2.

\begin{lemma}  For  constants $b\geq2$, $\eta>0$, if  $\psi: [0, +\infty)\rightarrow [0, +\infty)$ is  a
decreasing  function such that
$$-\eta\leq \psi^{'}(s)\leq 0$$
and
$$
A:=\int_0^{\infty}s^{b-1}\psi(s)ds>0,
$$
then, we have
\begin{equation}\label{eq:3.1}
\begin{aligned}
\int_0^\infty
s^{b+3}\psi(s)ds&\geq\dfrac{1}{b+4}(bA)^{\frac{b+4}{b}}\psi(0)^{-\frac{4}{b}}
+\frac{1}{3b(b+4)\eta^2}(bA)^{\frac{b+2}{b}}\psi(0)^{\frac{2b-2}{b}}\\
&\ \ \ +\frac{{(b+2)^2}}{72b(b+4)^2\eta^4}A\psi(0)^4
+\frac{q(b)}{\eta^6}(bA)^{\frac{b-2}{b}}\psi(0)^{\frac{6b+2}{b}}\\
&\geq\dfrac{1}{b+4}(bA)^{\frac{b+4}{b}}\psi(0)^{-\frac{4}{b}}+\frac{1}{3b(b+4)\eta^2}(bA)^{\frac{b+2}{b}}\psi(0)^{\frac{2b-2}{b}}\\
&
\ \ \ +\frac{{(b+2)^2}}{72b(b+4)^2\eta^4}A\psi(0)^4.
\end{aligned}
\end{equation}

where
$$
q(b)=\begin{cases}
\dfrac{(13b^3+56b^2-52b-32)(b+2)^3}{(12)^4b^4(b+4)^4},  & \quad \text{for $b\geq 4$ or $b=2$}, \\
& \\
\dfrac{(4b^3+11b^2-16b+4)(b+2)^3}{3\times (12)^3b^3(b+4)^3\eta^6}, &   \quad \text{ for $2<b<4$.} \\
\end{cases}
$$
\end{lemma}

\begin{proof}  By defining
$$\varphi(t)=\dfrac{\psi\big(\dfrac{\psi(0)}{\eta} t\big)}{\psi(0)},$$
we have $\varphi(0)=1$ and $-1\leq \varphi^{'}(t)\leq 0$.
Hence, without loss of generality, we can assume
$$\psi(0)=1\ \ \ {\rm and} \ \ \ \eta=1.$$
Define
\begin{equation}\label{eq:3.2}
D:=\int_0^\infty s^{b+3}\psi(s)ds.
\end{equation}
If $D=\infty$, the conclusion is correct. Hence, one can assume that
$$D=\int_0^\infty s^{b+3}\psi(s)ds<\infty.$$
 Thus,   $\lim\limits_{s\rightarrow \infty} s^{b+3}\psi(s)=0$ holds.
 Putting $h(s)=-\psi^{'}(s)$ for $s\geq 0$, we have
 $$
 0\leq h(s)\leq 1,\ \ \int_0^\infty h(s)ds=\psi(0)=1.
 $$
By making use of integration by parts, one has
\begin{equation}\label{eq:3.3}
\int_0^\infty s^bh(s)ds=b\int_0^\infty s^{b-1}\psi(s)ds=bA,
\end{equation}
\begin{equation}\label{eq:3.4}
\int_0^\infty s^{b+4}h(s)ds\leq (b+4)D
\end{equation}
since $\psi(s)\geq 0$.
By the  same assertion as in  \cite{[M]}, one can infer  that there exists an $\epsilon\geq0$ such that
\begin{equation}\label{eq:3.5}
\int_\epsilon^{\epsilon+1} s^bds=\int_0^\infty s^bh(s)ds=bA,
\end{equation}
\begin{equation}\label{eq:3.6}
\int_\epsilon^{\epsilon+1} s^{b+4}ds\leq\int_0^\infty s^{b+4}h(s)ds\leq(b+4)D.
\end{equation}
Since function $f(s)$ defined by
\begin{equation}\label{eq:3.7}
f(s)=bs^{b+4}-(b+4)\tau^4s^b+4\tau^{b+4}- 4\tau^{b+2}(s-\tau)^2,\ \ \ {\rm for}\ {\rm any} \ \tau>0,
\end{equation}
only has two critical points, one is  $s=\tau$, the other one is in the interval $(0, \tau)$,
we have $f(s)\geq 0$.
By  integrating the function $f(s)$ from $\epsilon$ to $\epsilon+1$, we deduce, from \eqref{eq:3.3} and \eqref{eq:3.4},
\begin{equation}\label{eq:3.8}
b(b+4)D-(b+4)\tau^4bA+4\tau^{b+4}\geq\dfrac13\tau^{b+2},\ \ \ {\rm for}\ {\rm any} \ \tau>0.
\end{equation}
Hence, we have, for  any  $\tau>0$,
\begin{equation}\label{eq:3.9}
\int_0^\infty s^{b+3}\psi(s)ds=D
\geq \frac{1}{b(b+4)}\big\{(b+4)\tau^4bA-4\tau^{b+4}+\dfrac13\tau^{b+2}\big\}.
\end{equation}
For $b\geq 4$ or $b=2$, we have, from Taylor expansion formula,
\begin{equation*}
\begin{aligned}
(1+t)^{\frac{4}b}\geq&   1+\dfrac4b t+\dfrac{2(4-b)}{b^2}t^2+\frac{2(4-b)(4-2b)}{3b^3}t^3\\
&+\frac{(4-b)(2-b)(4-3b)}{3b^4}t^4.
\end{aligned}
\end{equation*}

\begin{equation*}
(1+t)^{\frac{b+2}b}\geq 1+\dfrac{(b+2)}{b} t+\dfrac{(b+2)}{b^2}t^2+\frac{(b+2)(2-b)}{3b^3}t^3.
\end{equation*}
Since it is not hard to prove
\begin{equation}\label{eq:3.10}
\frac{1}{b+1}=\int_0^{1} s^bds\leq\int_0^\infty s^bh(s)ds=bA,
\end{equation}
by making use of the inequality $(s^b-1)(h(s)-\chi(s))\geq 0$ for $s\in [0,\infty)$, where
$\chi$ is the characteristic function of the interval $[0,1]$,
we have
$$
(b+1)bA\geq 1.
$$
Taking
\begin{equation*}
\tau=(bA)^{\frac1b}\big(1+\frac{b+2}{12(b+4)}(bA)^{\frac{-2}{b}}\big)^{\frac{1}{b}},
\end{equation*}
we have
\begin{equation}\label{eq:3.11}
\begin{aligned}
&(b+4)\tau^4bA-4\tau^{b+4}+\dfrac13\tau^{b+2}\\
=&(bA)^{1+\frac4b}\bigl(b-\dfrac{b+2}{3(b+4)}(bA)^{\frac{-2}{b}}\bigl)
\big(1+\frac{b+2}{12(b+4)}(bA)^{\frac{-2}{b}}\big)^{\frac{4}{b}}\\
&+\dfrac13(bA)^{1+\frac2b}
\big(1+\frac{b+2}{12(b+4)}(bA)^{\frac{-2}{b}}\big)^{\frac{b+2}{b}}.\\
\end{aligned}
\end{equation}
Putting
$$
t=\frac{b+2}{12(b+4)}(bA)^{\frac{-2}{b}},
$$
we derive, for $b\geq 4$ or $b=2$,
\begin{equation}\label{eq:3.12}
\begin{aligned}
&\bigl(b-\dfrac{b+2}{3(b+4)}(bA)^{\frac{-2}{b}}\bigl)
\big(1+\frac{b+2}{12(b+4)}(bA)^{\frac{-2}{b}}\big)^{\frac{4}{b}}\\
=&(b-4t)(1+t)^{\frac4b}\\
\geq &(b-4t)(1+\dfrac4b t+\dfrac{2(4-b)}{b^2}t^2+\dfrac{2(4-b)(4-2b)}{3b^3}t^3\\
&+\dfrac{(4-b)(2-b)(4-3b)}{3b^4}t^4)\\
=&b-\dfrac{2(4+b)}bt^2-\dfrac{4(4-b)(4+b)}{3b^2}t^3-\dfrac{(4-b)(2-b)(4+b)}{b^3}t^4\\
&-\dfrac{4(4-b)(2-b)(4-3b)}{3b^4}t^5\\
\geq & b-\dfrac{2(4+b)}bt^2-\dfrac{4(4-b)(4+b)}{3b^2}t^3-\dfrac{(4-b)(2-b)(4+b)}{b^3}t^4\\
=&b-\dfrac{2(4+b)}b\biggl(\frac{b+2}{12(b+4)}(bA)^{\frac{-2}{b}}\biggl)^2
-\dfrac{4(4-b)(4+b)}{3b^2}\biggl(\frac{b+2}{12(b+4)}(bA)^{\frac{-2}{b}}\biggl)^3\\
&-\dfrac{(4-b)(2-b)(4+b)}{b^3}\biggl(\frac{b+2}{12(b+4)}(bA)^{\frac{-2}{b}}\biggl)^4,\\
\end{aligned}
\end{equation}
\begin{equation}\label{eq:3.13}
\begin{aligned}
&
\big(1+\frac{b+2}{12(b+4)}(bA)^{\frac{-2}{b}}\big)^{\frac{b+2}{b}}\\
=&(1+t)^{\frac{b+2}{b}}\\
\geq &1+\dfrac{(b+2)}{b} t+\dfrac{(b+2)}{b^2}t^2+\frac{(b+2)(2-b)}{3b^3}t^3\\
=&1+\dfrac{2+b}b\biggl(\frac{b+2}{12(b+4)}(bA)^{\frac{-2}{b}}\biggl)+\dfrac{2+b}{b^2}\biggl(\frac{b+2}{12(b+4)}(bA)^{\frac{-2}{b}}\biggl)^2\\
&+\dfrac{(b+2)(2-b)}{3b^3}\biggl(\frac{b+2}{12(b+4)}(bA)^{\frac{-2}{b}}\biggl)^3.\\
\end{aligned}
\end{equation}
From \eqref{eq:3.11}, \eqref{eq:3.12} and \eqref{eq:3.13}, we obtain
\begin{equation}\label{eq:3.14}
\begin{aligned}
&(b+4)\tau^4bA-4\tau^{b+4}+\dfrac13\tau^{b+2}\\
\geq & b(bA)^{1+\frac4b}-\dfrac{2(4+b)}b\biggl(\frac{b+2}{12(b+4)}\biggl)^2(bA)\\
&-\dfrac{4(4-b)(4+b)}{3b^2}\biggl(\frac{b+2}{12(b+4)}\biggl)^3(bA)^{1-\frac{2}{b}},
\end{aligned}
\end{equation}
\begin{equation*}
\begin{aligned}
&-\dfrac{(4-b)(2-b)(4+b)}{b^3}\biggl(\frac{b+2}{12(b+4)}\biggl)^4(bA)^{1-\frac{4}{b}}\\
&+\dfrac13(bA)^{1+\frac2b}+\dfrac{2+b}{3b}\biggl(\frac{b+2}{12(b+4)}\biggl)(bA)
+\dfrac{2+b}{3b^2}\biggl(\frac{b+2}{12(b+4)}\biggl)^2(bA)^{1-\frac2b}\\
&+\dfrac{(b+2)(2-b)}{9b^3}\biggl(\frac{b+2}{12(b+4)}\biggl)^3(bA)^{1-\frac4b}\\
=&b(bA)^{1+\frac4b}+\dfrac13(bA)^{1+\frac2b}
+\frac{1}{72}\dfrac{(b+2)^2}{b(b+4)}(bA)\\
&+\dfrac{4(b-1)(b+4)}{3b^2}\biggl(\frac{b+2}{12(b+4)}\biggl)^3(bA)^{1-\frac2b}\\
&+\dfrac{(b^2-4)(8-3b)}{36b^3}\biggl(\frac{b+2}{12(b+4)}\biggl)^3(bA)^{1-\frac4b}.\\
\end{aligned}
\end{equation*}
From $b\geq 4$ or $b=2$ and \eqref{eq:3.10}, we have
$$
(bA)^{\frac2b}\geq \dfrac1{(b+1)^{\frac2b}}\geq\dfrac13
$$
since $(b+1)^{\frac2b}\leq 3$,  and
\begin{equation}\label{eq:3.15}
\dfrac{4(b-1)(b+4)}{3b^2}+\dfrac{(b^2-4)(8-3b)}{36b^3}(bA)^{-\frac2b}
\geq \frac{13b^3+56b^2-52b-32}{12b^3}.
\end{equation}
According \eqref{eq:3.9}, \eqref{eq:3.14} and \eqref{eq:3.15},  we obtain
\begin{equation*}
\begin{aligned}
&\int_0^\infty s^{b+3}\psi(s)ds=D\\
\geq& \frac{1}{b(b+4)}\big\{(b+4)\tau^4bA-4\tau^{b+4}+\dfrac13\tau^{b+2}\big\}\\
\geq& \frac{1}{b(b+4)}\biggl\{b(bA)^{1+\frac{4}{b}}+\frac{1}{3}(bA)^{1+\frac{2}{b}}+
\frac{1}{72}\frac{(b+2)^2}{b(b+4)}(bA)\\
&+\frac{13b^3+56b^2-52b-32}{12b^3}\biggl(\frac{b+2}{12(b+4)}\biggl)^3
(bA)^{1-\frac{2}{b}}\biggl\}.
\end{aligned}
\end{equation*}
For $2<b<4$, we can obtain the following inequality using the same method as the case of $b\geq4$,
\begin{equation*}
\begin{aligned}
& \int_0^\infty s^{b+3}\psi(s)ds=D\\
\geq& \frac{1}{b(b+4)}\biggl\{b(bA)^{1+\frac{4}{b}}+\frac{1}{3}(bA)^{1+\frac{2}{b}}+
\frac{1}{72}\frac{(b+2)^2}{b(b+4)}(bA)\\
&+\frac{12b^3+33b^2-48b+12}{9b^3}\biggl(\frac{b+2}{12(b+4)}\biggl)^3
(bA)^{1-\frac{2}{b}}\biggl\}.\\
\end{aligned}
\end{equation*}
This finishes the proof of the lemma 3.1.
\end{proof}

\vskip 3mm
\noindent
{\it Proof of  Theorem 1.2}. Using the same notations as those of \cite{[CW]}, we can obtain from Lemma 3.1 that
\begin{equation}\label{eq:3.16}
\aligned
\sum_{j=1}^k\Gamma_j\geq& nB_n\int_0^\infty s^{n+3}g(s)ds\\
    \geq& \frac{n}{n+4}(B_n)^{-\frac{4}{n}}k^{\frac{n+4}{n}}g(0)^{-\frac{4}{n}}
    +\frac{1}{3(n+4)\eta^2}k^{\frac{n+2}{n}}(B_n)^{-\frac{2}{n}}g(0)^{\frac{2n-2}{n}}\\
   &+\frac{(n+2)^2}{72n(n+4)^2\eta^4}k g(0)^4,
\endaligned
\end{equation}
where $g: [0, +\infty)\rightarrow
[0, (2\pi)^{-n}\text{vol}(\Omega)]$ is a non-increasing function of $|x|$ and $g(x)$ is defined by $g(|x|):=h^*(x)$. Here $h^*$ is the symmetric decreasing rearrangement of $h$, $h$ is defined by $h(z):=\sum_{j=1}^k|\widehat{\varphi}_j(z)|^2$, $\widehat{\varphi}_j(z)$ is the Fourier transform of the
trial function $\varphi_j(x)$,
\begin{equation*}
   \varphi_j(x)={\begin{cases}
    u_j(x), & \ \  x\in \Omega ,\\
     0 , & \ \ x\in \mathbf{R}^n\setminus\Omega, \\
         \end{cases}}
\end{equation*}
here $u_j$ is an orthonormal eigenfunction
corresponding to the eigenvalue $\Gamma_j$.

Now defining a function $F$ by
\begin{equation}\label{eq:3.17}
\aligned
F(t)=&\frac{n}{n+4}(B_n)^{-\frac{4}{n}}k^{\frac{n+4}{n}}t^{-\frac{4}{n}}
   +\frac{1}{3(n+4)\eta^2}k^{\frac{n+2}{n}}(B_n)^{-\frac{2}{n}}t^{\frac{2n-2}{n}}\\
   &+\frac{(n+2)^2}{72n(n+4)^2\eta^4}k t^4.
\endaligned
\end{equation}

Since  $\eta\geq(2\pi)^{-n}B_n^{-\frac{1}{n}}\text{vol}(\Omega)^{\frac{n+1}{n}}$, we obtain
\begin{equation}\label{eq:3.18}
\aligned
F^{'}(t)\leq&  -\dfrac{4}{n+4}(B_n)^{-\frac{4}{n}}k^{\frac{n+4}{n}}t^{-1-\frac{4}{n}}\\
&+\dfrac{2(n-1)}{3n(n+4)}k^{\frac{n+2}{n}}(2\pi)^{2n}\text{vol}(\Omega)^{-\frac{2(n+1)}{n}} t^{\frac{n-2}{n}}\\
    & +\dfrac{(n+2)^2}{18n(n+4)^2}k
    t^3(2\pi)^{4n}(B_n)^{\frac{4}{n}}\text{vol}(\Omega)^{-\frac{4(n+1)}{n}}\\
    =&\dfrac{k}{n+4}t^{-\frac{n+4}{n}}\times
    \biggl\{\frac{2(n-1)}{3n}(2\pi)^{2n}k^{\frac{2}{n}}
    \text{vol}(\Omega)^{-\frac{2(n+1)}{n}}t^{\frac{2n+2}{n}}\\
    &-4(B_n)^{-\frac{4}{n}}k^{\frac{4}{n}}
    +\dfrac{(n+2)^2}{18n(n+4)}(2\pi)^{4n}(B_n)^{\frac{4}{n}}\text{vol}(\Omega)^{-\frac{4(n+1)}{n}}t^{\frac{4n+4}{n}}
    \biggl\}.
\endaligned
\end{equation}
Hence, we have
\begin{equation}\label{eq:3.19}
\aligned
& \frac{n+4}{k}t^{\frac{n+4}{n}}F^{'}(t)\\
\leq &\frac{2(n-1)}{3n}(2\pi)^{2n}k^{\frac{2}{n}}
    \text{vol}(\Omega)^{-\frac{2(n+1)}{n}}t^{\frac{2n+2}{n}}\\
    & -4(B_n)^{-\frac{4}{n}}k^{\frac{4}{n}}
    +\frac{(n+2)^2}{18n(n+4)}(2\pi)^{4n}(B_n)^{\frac{4}{n}}\text{vol}(\Omega)^{-\frac{4(n+1)}{n}}t^{\frac{4n+4}{n}}.
\endaligned
\end{equation}
Since the right hand side of \eqref{eq:3.19}  is an increasing function of $t$,
if the right hand side of \eqref{eq:3.19} is not larger than $0$ at $t=(2\pi)^{-n}\text{vol}(\Omega)$, that
is,
\begin{equation}\label{eq:3.20}
\aligned & \frac{2(n-1)}{3n}(2\pi)^{2n}k^{\frac{2}{n}}
    \text{vol}(\Omega)^{-\frac{2(n+1)}{n}}((2\pi)^{-n}\text{vol}(\Omega))^{\frac{2n+2}{n}}\\
    &+\frac{(n+2)^2}{18n(n+4)}
(2\pi)^{4n}(B_n)^{\frac{4}{n}}\text{vol}(\Omega)^{-\frac{4(n+1)}{n}}((2\pi)^{-n}\text{vol}(\Omega))^{\frac{4n+4}{n}} -4(B_n)^{-\frac{4}{n}}k^{\frac{4}{n}}\\
\leq&0,
\endaligned
\end{equation}
then one has from \eqref{eq:3.19} that $F^{'}(t)\leq0$ on $(0,
(2\pi)^{-n}\text{vol}(\Omega)]$. If $F^{'}(t)\leq0$, then $F(t)$ is a decreasing function on $(0,
(2\pi)^{-n}\text{vol}(\Omega)]$. By a direct calculation, we have
that \eqref{eq:3.20} is equivalent to
\begin{equation}\label{eq:3.21}
\frac{(n-1)}{6n}(2\pi)^{-2}k^{\frac{2}{n}}
+\frac{(n+2)^2}{72n(n+4)}(2\pi)^{-4}(B_n)^{\frac{4}{n}}\leq (B_n)^{-\frac{4}{n}}k^{\frac{4}{n}}.
\end{equation}
We now check the equation \eqref{eq:3.21}. Note that $(2\pi)^{-2}(B_n)^{\frac{4}{n}}<1$, then one has
\begin{equation}\label{eq:3.22}
\aligned & \frac{(n-1)}{6n}(2\pi)^{-2}k^{\frac{2}{n}}+\frac{(n+2)^2}{72n(n+4)}(2\pi)^{-4}(B_n)^{\frac{4}{n}}\\
<&\frac{1}{6}(2\pi)^{-2}k^{\frac{2}{n}}+\frac{1}{36}(2\pi)^{-2}
<(2\pi)^{-2}\left\{\frac{1}{6}k^{\frac{4}{n}}+\frac{1}{36}\right\}\\
<&(2\pi)^{-2}k^{\frac{4}{n}}< (B_n)^{-\frac{4}{n}}k^{\frac{4}{n}}.
\endaligned
\end{equation}

On the other hand, since $0<g(0)\leq(2\pi)^{-n}\text{vol}(\Omega)$ and the
right hand side of the formula \eqref{eq:3.16} is $F(g(0))$, which is a
decreasing function of $g(0)$ on $(0, (2\pi)^{-n}\text{vol}(\Omega)]$, then
we can replace $g(0)$ by $(2\pi)^{-n}\text{vol}(\Omega)$ in \eqref{eq:3.16} which
gives inequality
$$
\aligned \frac{1}{k}\sum_{j=1}^k\Gamma_j\geq&
\frac{n}{n+4}\dfrac{16\pi^4}{\big(B_n\text{vol}(\Omega)\big)^{\frac{4}{n}}}k^{\frac{4}{n}}\\
 &+\frac{n+2}{12n(n+4)}\frac{\text{vol}(\Omega)}{I(\Omega)}\frac{n}{n+2}
 \dfrac{4\pi^2}{\big(B_n\text{vol}(\Omega)\big)^{\frac{2}{n}}}k^{\frac{2}{n}}\\
 &+\frac{(n+2)^2}{1152n(n+4)^2}
 \left(\frac{\text{vol}(\Omega)}{I(\Omega)}\right)^2.
\endaligned
$$
This completes the proof of Theorem 1.2.
$$\eqno{\Box}$$

\end {document}